\documentclass{article}
\usepackage{amsmath}
\usepackage{amsthm}
\usepackage{amssymb}
\usepackage{amsfonts}

\usepackage[english]{babel}
\usepackage{wrapfig}
\usepackage{graphicx}
\usepackage{geometry}
\usepackage[rightcaption]{sidecap}
\usepackage{tikz}
\usepackage{bbm}
\usepackage{url}
\usepackage{marvosym}
\usepackage{wasysym}
\usepackage{pifont}
\usepackage{mathrsfs}
\usepackage{caption}
\DeclareMathAlphabet{\mathpzc}{OT1}{pzc}{m}{it}

\allowdisplaybreaks[2]
\author{Jingchen Hu}

\newtheorem{theorem}{Theorem}[section]

\newtheorem{lemma}[theorem]{Lemma}

\numberwithin{equation}{section}
\newtheorem*{oldtheorem}{Theorem}
\newtheorem{condition}{Condition}
\newcommand{\ER}{\mathbb{R}}

\newcommand{\EC}{\mathbb{C}}

\newcommand{\ijbar}{{i\overline j}}

\newcommand{\partialijbar}{\partial_{i\overline j}}
\newcommand{\bfs}{{\bf s}}

\newcommand{\jbar}{{\overline{j}}}

\newcommand{\tbar}{{\overline{t}}}

\newcommand{\Abar}{{\overline{A}}}
\newcommand{\Bbar}{{\overline{B}}}

\newcommand{\Abarinverse}{\overline{A^{-1}}}

\newcommand{\Bjbar}{B_{\jbar}}

\newcommand{\Ainverse}{{A^{-1}}}

\newcommand{\BBi}{{\mathscr{B}^{(i)}}}
\newcommand{\BBj}{{\mathscr{B}^{(j)}}}

\newcommand{\BBjbar}{{\left(\BBj\right)^\ast}}

\newcommand{\qbar}{{\overline{q}}}

\newcommand{\partiali}{\partial_i}

\newcommand{\partialjbar}{\partial_{\overline j}}

\newcommand{\uijbar}{{u^{\ijbar}}}

\title{Convexity of the Potential Function of the Einstein-K\"ahler Metric on a Convex Domain}

\author{Jingchen Hu, Li Sheng}
\begin{document}
\maketitle

\begin{abstract}
Suppose that $u$ is the potential function of a complete K\"ahler-Einstein metric on a bounded strictly convex domain in $\EC^n$. We prove that $u$ itself is strictly convex.
\end{abstract}

\section{Introduction}
Suppose that $u$ is the potential function of a complete K\"ahler-Einstein metric on a strictly pseudo-convex domain in $\EC^n$. Then $u$ satisfies
\begin{align}
	\label{eq:potential_KE}
	\left\{
 	\begin{array}{cc}
 	\det(u_{i\jbar})=e^{(n+1)u}& \text{in\ \ } \Omega;\\
 	u(x)\rightarrow+\infty & \text{as\ \ }x\rightarrow \partial\Omega.
 \end{array}
	\right.
\end{align}
Let $v=e^{-u}$, then $v$ satisfies
\begin{align}
	\label{eq:potential_KE_exp}
	\left\{
	\begin{array}{cc}
		\det
		\left(
		\begin{array}{cc}
			v& v_{\jbar}\\
			v_i& v_{i\jbar}
			\end{array}
		\right)=(-1)^n, & \text{in\ \ }\Omega;\\
		v=0,& \text{on \ \ }\partial\Omega.
	\end{array}
	\right.
\end{align}
By studying the equations above, Cheng and Yau\cite{ChengYau_80_KahlerMetric} proved the following result:
\begin{oldtheorem}[Corollary 6.6 of \cite{ChengYau_80_KahlerMetric}]
	Suppose that $\Omega$ is a $C^k$, $k\geq \max(3n+6,2n+9)$, strictly pseudoconvex domain in $\EC^n$. Then there is a function $v\in C^{\omega}(\Omega)\cap C^{n+3/2-\delta}(\overline{\Omega})$, for any $\delta>0$, solving (\ref{eq:potential_KE_exp}).
\end{oldtheorem}

It's interesting to observe that the regularity of $v$ is less than $C^{n+3/2}$, regardless of $k$, which has been observed by Fefferman in an earlier work \cite{Fefferman1976Ann}. Fefferman suggested that logarithmic terms exist following the $(n+1)$-th order term in the formal asymptotic expansion. This is proved by Lee and Melrose, in \cite{Lee_Melrose_Acta}. Suppose that $\Omega$ is smooth and $\varphi$ is a smooth defining function of $\Omega$. Lee and Melrose proved that
\begin{align}
	u+\log(-\varphi)-\sum_{j=0}^N a_j \varphi^{(n+1)j}(\log(-\varphi))^j\in C^{(n+1)N-1}(\overline{\Omega}),
\end{align}where ${a_j}'s$ are smooth functions in $\overline\Omega$.
The convergence of the above asymptotic expansion is proved by \cite{Han_Jiang_Expansion}, using a method developed in \cite{HanJiang_MinimalSurface}, where a similar boundary regularity obstacle was studied for the minimal surface problem in hyperbolic spaces.

In this paper, by extending the computation of \cite{Hu22Nov} and \cite{HuCConvexity} to non-degenerate complex Monge-Amp\`ere equations, we prove the following result:
\begin{theorem}
	Suppose that $\Omega$ is a $C^k$, $k\geq \max(3n+6,2n+9)$, bounded strictly convex domain in $\EC^n$. Then the solution $u$ of  (\ref{eq:potential_KE}) is a strictly convex function.
\end{theorem}

In (\ref{eq:potential_KE}), if the complex Monge-Amp\`ere operator is replaced by the Laplacian operator, the corresponding result is a consquence of the constant rank theorem of \cite{CaffarelliDuke1985} and \cite{Korevaar_Lewis}.  

Constant rank theorems for nonlinear equations were proved by \cite{Caffarelli_Guan_Ma_constantrank}, \cite{BianGuan09_Invention} and \cite{Garbo_Weinkove_ConstantRank}. In \cite{GuanPhong12MaxRank}, a constant rank theorem for degenerate complex Monge-Amp\`ere equation was proved in the case of low dimension. The general dimensional case remained as a conjecture until it was proved by the first author of this paper in \cite{Hu22Nov}.  In \cite{LQun_Complex}, a constant rank theorem regarding complex Hessian was proved.

Readers may wonder if the result of this paper could be proved by directly applying the existing constant rank theorems (for example, Theorem 1 of \cite{Caffarelli_Guan_Ma_constantrank}). To use the existing constant rank theorems we need to 
verify some structure conditions, in particular, an ``inverse convexity" condition. Let $u$ be a strictly convex function on $\EC^n$; let $H$ be the real Hessian of $u$, and $A$ be the complex Hessian of $u$.  Denote the map 
\begin{align}
H\ \rightarrow \det{A}
\end{align} by $F$. Constant rank theorems of  \cite{Caffarelli_Guan_Ma_constantrank}, \cite{BianGuan09_Invention} and \cite{Garbo_Weinkove_ConstantRank} require that
\begin{condition}[Inverse Convexity Condition]
	\begin{align}F(H^{-1}) \text{  \   is a locally convex function of $H$.}
	\end{align}
	\end{condition} We believe that the above condition holds and that $\log(F(H^{-1}) )$ is a locally convex function of 
	$H$; however, we do not yet have a proof for this.
	
	In a couple of works \cite{Hu22Nov} and \cite{HuCConvexity}, we developed a computation technique which is particularly useful and convenient for the convexity estimate of solutions to complex Monge-Amp\`ere equations.  But in the previous works, we only used this method for homogenous complex Monge-Amp\`ere equation. We later discovered, the method can be applied to more general complex Monge-Amp\`ere equaions. 
	
	Besides the convexity of the solution, there are many researches on the convexity of level sets of the solution or power(or logorithmic and exponential) convexity of the solutions \cite{BianGuanXuMa} \cite{GuanXu_convexity} \cite{Xinan_Ou_Summary_English} \cite{MXu_PowerConvexity} \cite{XinanMa_Chinese} \cite{Shishujun_Green_LevelSets}. Our computation is useful in these settings too. In an upcomming paper \cite{ChenHongyu}, we prove the power convexity of solutions to non-degenerate complex Monge-Amp\`ere equations in a convex domain.

\section{Proof of the Main Theorem}
   Let $A=(u_{\ijbar})$,  $B=(u_{ij})$ and
   \begin{align}
   	M=A-B\Abarinverse \ \Bbar.
   \end{align}
   According to the estimate of \cite{ChengYau_80_KahlerMetric}, we already know that the minimum eigenvalue of  $A$ has a lower bound. So using Lemma \ref{lemma} we only need to show that $M$ is positive definite in $\Omega$. 
   
   In the following, we show $u^{\ijbar}{\partial_{\ijbar}} M-(n+1)M$ is non-positive definite as a matrix, then using maximum principle we derive $M>0$ in $\Omega$.  In subsection \ref{sec:eq_for_AB}, we derive equations for $A, B$ to simplify the computation.  
   In subsection \ref{sec:computing_LM} we compute $u^{\ijbar}{\partial_{\ijbar}} M$. In subsection \ref{sec:Proving_Result_Maximum_Principle}, utalizing the results above, we prove the main theorem.
\subsection{Equations for Second Derivatives of $u$: $A, B$}
\label{sec:eq_for_AB}
The logrithm of equation (\ref{eq:potential_KE}) is
\begin{align}
	\log\det(u_{\ijbar})=(n+1)u
	\label{eq:2.2}
\end{align}
We apply $\partial_p$ to the equation above and get
\begin{align}
	u^{\ijbar}u_{p\ijbar}=(n+1)u_p.
	\label{eq:2.3}
\end{align}
Then we apply $\partial_{\qbar}$ to (\ref{eq:2.3}) and get
 \begin{align}
 	u^{\ijbar}{u_{p\qbar\ijbar}}=u^{i\overline t}u_{s\overline t \qbar} u^{s\jbar} u_{p\ijbar}+(n+1) u_{p\qbar}.
 	\label{eq:2.4}
 \end{align}
 On the right-hand side of (\ref{eq:2.4}), we switch indices $i$ and $s$ and get
  \begin{align}
 	u^{\ijbar}{u_{p\qbar\ijbar}}=u^{s\tbar}u_{i\tbar\qbar}u^{\ijbar} u_{ps\jbar}+(n+1) u_{p\qbar}.
 	\label{eq:2.4+}
 \end{align}  Using the notation of matrix multiplication, we can simplify (\ref{eq:2.4+}) as
 \begin{align}
 	u^{\ijbar} \partial_{\ijbar}A=u^{\ijbar}\partialjbar B \Abarinverse \partiali\overline{B}+(n+1)A.
 	\label{eq:DDA_BB}
 \end{align}
 Similarly switching indices  $j$  and $t$ in the  right-hand side of (\ref{eq:2.4}) gives
  \begin{align}
 	u^{\ijbar}{u_{p\qbar\ijbar}}=u^{i\overline j}u_{s\overline j \qbar} u^{s\tbar} u_{p i\tbar}+(n+1) u_{p\qbar}.
 	\label{eq:2.4++}
 \end{align} The above can be simplified to the following, using the notation of matrix multiplication:
 \begin{align}
 	u^{\ijbar} \partial_{\ijbar}A=u^{\ijbar}\partiali A \Ainverse \partialjbar A+(n+1)A.
 \label{eq:DDA_AA}
 \end{align}
Applying $\partial_q$ to (\ref{eq:2.3}) gives
 \begin{align}
	u^{\ijbar}{u_{pq\ijbar}}=u^{i\overline t}u_{s\overline t q} u^{s\jbar} u_{p\ijbar}+(n+1) u_{pq}.
	\label{eq:2.4_pq_nobar}
\end{align}
Similar to the treatment of (\ref{eq:2.4}), we can simplify it using matrix multiplication: (\ref{eq:2.4_pq_nobar}) is equivalent to 
\begin{align}
	 	u^{\ijbar} \partial_{\ijbar}B=u^{\ijbar}\partiali A \Ainverse \partialjbar B+(n+1)B,
	\label{eq:DDB_AB}
\end{align} and
\begin{align}
	    u^{\ijbar} \partial_{\ijbar}B=u^{\ijbar}\partialjbar B \Abarinverse \partiali \Abar+(n+1)B.
	\label{eq:DDB_BA}
\end{align}
\subsection{Computing $u^{\ijbar}{\partial_{\ijbar}} M$}
\label{sec:computing_LM}
We want to apply $u^{\ijbar}\partial_{\ijbar}$ to the matrix $M$. To simplify the computation, let
\begin{align}
	\BBi=\partial_iB -B\Abarinverse \partial_i\Abar.
\end{align}
Then $\BBi$ satisfies
\begin{align}
	\partial_i(B\Abarinverse)=\BBi\Abarinverse.
\end{align}
The  conjugate transpose of the relation above is 
\begin{align}
	\partial_\jbar(\Abarinverse\ \Bbar)=\Abarinverse\left(\BBj\right)^\ast.
\end{align} Here we replace the index $i$ by $j$.

Then direct computation gives
\begin{align}
	&u^{\ijbar}\partial_{\ijbar}(A-B\Abarinverse\ \Bbar)\\
	=&u^{\ijbar}\partialjbar\left(\partiali A-\BBi\Abarinverse\ \Bbar-B\Abarinverse\partiali\Bbar\right)\\
	=&\uijbar A_{\ijbar}-\uijbar\partialjbar\BBi\Abarinverse\ \overline{B}-\uijbar\BBi\Abarinverse\BBjbar
	   -\uijbar \partial_\jbar B\Abarinverse\partial_i\Bbar-\uijbar B\partial_\jbar\left(\Abarinverse \partial_i\Bbar\right).
	\label{eq:4.3Lastline}
\end{align}
In (\ref{eq:4.3Lastline}), the third term  is a non-positive definite matrix. 
For the first term and the forth term, we combine them together and use relation (\ref{eq:DDA_BB}):
\begin{align}
	u^{\ijbar}\left(A_{\ijbar}-\partialjbar B\Abarinverse \partiali\Bbar\right)=(n+1)A.
	\label{eq:partial_Res_1}
\end{align}

In (\ref{eq:4.3Lastline}), it remains to study the second term and the last term. In the following we compute them seperately.

{\bf(i) Computing $\partialjbar\BBi$}
\begin{align}
  &\uijbar\partialjbar\BBi
  \label{eq:partialjbarBBi}\\
=&\uijbar\partialjbar\left( \partiali B-B\Abarinverse\partiali\Abar\right)\\
=&
\left(\partialijbar B-\Bjbar\Abarinverse\partiali \Abar\right)\uijbar
-B\uijbar\partialjbar \left( \Abarinverse\partiali\Abar\right)
\label{eq:2.20}
\end{align}

For the second term of (\ref{eq:2.20}), we compute $\uijbar\partialjbar \left( \Abarinverse\partiali\Abar\right)$ and then use the complex conjugate of the relation (\ref{eq:DDA_AA}). Direct computation gives
\begin{align}
  \uijbar	\partialjbar \left( \Abarinverse\partiali\Abar\right)
=\Abarinverse\uijbar\left( \partialijbar \Abar-\partialjbar\Abar\ \Abarinverse\partiali\Abar\right)
\end{align}
Then using the complex conjugate of the relation (\ref{eq:DDA_AA}), we know the above equals to $(n+1)I$.

For the first term of (\ref{eq:2.20}), we use the relation (\ref{eq:DDB_BA}). Then we find that  in (\ref{eq:2.20}) the first term and the second term both equal to $(n+1)B$ so they cancel with each other:
\begin{align}
(\ref{eq:partialjbarBBi})=(n+1)B-(n+1)B=0.
\label{eq:partial_Res_2}
\end{align}
{\bf(ii) Computing $\partial_\jbar\left(\Abarinverse \partial_i\Bbar\right)$}

Direct computation gives
\begin{align}
	\partial_\jbar\left(\Abarinverse \partial_i\Bbar\right)
	=\Abarinverse \left(\partialijbar \Bbar -\partial_\jbar\Abar\ \Abarinverse\partiali\Bbar\right).
	\label{eq:partialjbarAinversbarpartialiBbar}
\end{align}
Then using the complex conjugate of the relation (\ref{eq:DDB_AB}), we find
\begin{align}
	(\ref{eq:partialjbarAinversbarpartialiBbar})=(n+1)\Abarinverse \ \Bbar.
	\label{eq:partial_Res_3}
\end{align}
Combining (\ref{eq:partial_Res_1}), (\ref{eq:partial_Res_2}) and (\ref{eq:partial_Res_3}) all together, we find
\begin{align}
\uijbar	\partialijbar M-(n+1)M=-\BBi\Abarinverse\BBjbar
\end{align}
is a non-positive definite matrix.
\subsection{Proving the Result Using the Maximum Principle}
\label{sec:Proving_Result_Maximum_Principle}
Let $\bfs=s^i \partial_i$ be a constant vector field and 
\begin{align}
	m=M_{p\qbar}s^p s^{\qbar}.
\end{align}
The result of the previous section implies
\begin{align}
	\uijbar\partialijbar m\leq (n+1)m.
\end{align}
Corollary 6.6 of \cite{ChengYau_80_KahlerMetric} says that $v\in C^2(\overline{\Omega})$. Then equation (\ref{eq:potential_KE_exp}) implies $\nabla v\neq 0$  on $\partial\Omega$ since $v=0$ on $\partial\Omega$. Because $v=0$ on $\partial\Omega$, $\Omega$ is strictly convex and $\nabla v\neq0$ in a neighborhood of $\partial\Omega$, the level sets of $v$ are strictly convex in a neighborhood of $\partial\Omega$. From this, we can prove that in a neighborhood of $\partial \Omega$ $u$ is convex. Since $u=-\log(v)$ direct computation gives
\begin{align}
  D^2_{\ER}u=\frac{1}{v}	\left(-D^2_{\ER}v+\frac{1}{v}\nabla v\otimes\nabla v\right).
\end{align}
Because $-D^2_{\ER}v$ is positive definite in the tangential direction of level sets adding $\frac{1}{v}\nabla v\otimes\nabla v$ makes $-D^2_{\ER}v+\frac{1}{v}\nabla v\otimes\nabla $ positive definite if $v$ is small enough. Therefore for $r$ small enough
\begin{align}
	D_{\ER}^2u   \text{\ is positive definite  in \ \ \ \ } \Omega_r=\left\{
	z\in \Omega\ |\ 0<\text{dist}(z,\partial \Omega)<r
	\right\}.
\end{align} 
Then Lemma \ref{lemma} implies 
\begin{align}
	M   \text{\ is positive definite  in \ \ \ \ } \Omega_r=\left\{
	z\in \Omega\ |\ 0<\text{dist}(z,\partial \Omega)<r
	\right\}.
\end{align} 
So for any vector field $\bfs$ the corresponding $m>0$ in $\Omega_r$. Therefore the maximum principle (Corollary 3.2 of \cite{GT}) implies $m>0$ in $\Omega$. It follows that $M$ is positive definite in $\Omega$. Again using Lemma \ref{lemma}, we conclude that  $u$ is strictly convex in $\Omega$.

\appendix
\section{Linear Algebra Lemma}
In this paper, we need to estimate the real convexity of a function $u$ on $\EC^n$ using complex derivatives $u_{\ijbar}$, $u_{ij}$. The following lemma provides a tool for switching between complex and real contexts. The proof is almost the same as that of Lemma A6 of \cite{Hu22Nov}.

\begin{lemma}
	\label{lemma}
	 $u$ is a $C^2$ function on $\EC^n$, with coordinates $z_i=x_i+\sqrt{-1}y_i$, $i=1,2,\ ...\ ,n$. Denote the real Hessian of $u$ by 
	  \begin{align}
	  	  H=
	  	  \left(\begin{array}{cc}  U&V\\
	  	  V^T& W
	  	  \end{array}
	  	  \right)
	  	  =
	  	  \left(\begin{array}{cc}  u_{x_i x_j}& u_{x_i y_j}\\
	  	  	                                         u_{y_i x_j}& u_{y_i y_j}
	  	  	          \end{array}
	  	  \right).
	  \end{align}
	  Denote that 
	  \begin{align}
	  	A=\left(
	  	   u_{\ijbar}
	  	\right), \ \ \ \ \ B=\left(
	  	u_{ij}
	  	\right).
	  \end{align}
	  Then $u$ is strictly convex (i.e. $H>0$) if and only if $A>0$ and $A-B\overline{A^{-1}}\ \overline B>0$. In above, a matrix $>0$ means that it's positive definite.
\end{lemma}

\begin{proof}
	We only need to show that at a given point $H>0$  if and only if $A>0$ and $A-B\overline{A^{-1}}\ \overline B>0$. This becomes  obvious once we derive the relation between $H$ and $A, B$.
	
	Direct computation gives
	\begin{align}
		u_{\ijbar}=\frac{1}{4}\left(u_{x_ix_j}+u_{y_iy_j}+\sqrt{-1}(u_{x_iy_j}-u_{y_ix_j})\right);\\
			u_{ij}=\frac{1}{4}\left(u_{x_ix_j}-u_{y_iy_j}-\sqrt{-1}(u_{x_iy_j}+u_{y_ix_j})\right).
	\end{align}Using the matrix notation, we have
	\begin{align}
		\label{matrix_Relation_A}
		A=\frac{1}{4}(U+W+\sqrt{-1}(V-V'));\\
				\label{matrix_Relation_B}
		B=\frac{1}{4}(U-W-\sqrt{-1}(V+V')).
	\end{align}
	  Let
	  \begin{align}
	  	Q=\left(
	  	\begin{array}{cc}
	  	A	&  B    \\
	  	\overline{B}	&\overline{A}
	  	\end{array}
	  	\right).
	  \end{align}
	  $Q=Q^\ast$ since $B$ is symmetric. 
	  Then (\ref{matrix_Relation_A}) and (\ref{matrix_Relation_B}) gives
	  \begin{align}
	  	\label{QH}
	  	Q=\left(
	  	\begin{array}{cc}
	  	I	&  -\sqrt{-1}I   \\
	  	I	&   \sqrt{-1}I
	  	\end{array}
	  	\right)  H
	  	\left(
	  	\begin{array}{cc}
	  		I	& I   \\
	  	 \sqrt{-1}	I	&  - \sqrt{-1}I
	  	\end{array}
	  	\right).
	  \end{align}
	  This implies that $H>0$ if and only if $Q>0$. When $A$ is invertible, we can transform $Q$ so that the relation between $Q$ and $A-B\overline{A^{-1}}\ \overline{B}$ is more obvious. We have
	  \begin{align}
	  	\left(\begin{array}{cc}
	  		A-B\overline{A^{-1}}\overline{B}& 0\\
	  		0& \overline{A}
	  	\end{array}\right)
	  	=\left(\begin{array}{cc}
	  		I& -B\overline{A^{-1}}\\
	  		0& I
	  	\end{array}\right)
	  	Q
	  	\left(\begin{array}{cc}
	  		I& 0\\
	  		-\overline{A^{-1}}B^\ast& I
	  	\end{array}\right).
	  	\label{Qregularize}
	  \end{align}
	  If $A$ and $A-B\overline{A^{-1}}\overline{B}$ are both positive definite, using (\ref{Qregularize}) we know $Q$ is positive definite. If $Q$ is positive definite, it follows that 
	  $A$, being a principal submatrix of 
	 $Q$, is also positive definite.Then equation (\ref{Qregularize}) implies that 
	  $A-B\overline{A^{-1}}\overline{B}$ is positive definite.
\end{proof}

\ \ Jingchen Hu\\
Sichuan University\\jingchenhu@scu.edu.cn\\

Li Sheng\\
Sichuan University\\ lsheng@scu.edu.cn

\begin{thebibliography}{99}


	\bibitem[BG09]{BianGuan09_Invention}B. Bian, P. Guan, {\em A microscopic convexity principle for nonlinear partial differential
	equations.} Invent. Math. 177 (2009), 307–335.
	\bibitem[BGMX11]{BianGuanXuMa} B. Bian, P. Guan, L. Xu, X. Ma, {\em A Constant Rank Theorem for Quasiconcave Solutions of Fully Nonlinear Partial Differential Equations.} Indiana University Mathematics Journal, Vol. 60, No. 1 (2011), pp. 101-119.
	
	

\bibitem[C85]{CaffarelliDuke1985} L, Caffarelli, A. Friedman. {\em Convexity of Solutions of Semilinear Elliptic Equations. Duke Mathematical Journal}, 1985, Vol.52(2), 431–456.
\bibitem[CGM07]{Caffarelli_Guan_Ma_constantrank}L.Caffarelli, P. Guan, X. Ma, {\em A Constant Rank Theorem for Solutions of Fully Nonlinear Elliptic Equations. } Communications on Pure and Applied Mathematics, 2007, Vol. 60(12), 1769–1791.
\bibitem[CM18]{ChuanQiangMa} C. Chen, X. Ma, {\em The microscopic convexity of level sets of solutions for elliptic
	and parabolic equations}. Scientia Sinica Mathematica, Volume 48, Issue 10: 1205 (2018). (in Chinese)

\bibitem[CY80]{ChengYau_80_KahlerMetric} S.Y. Cheng, S.T. Yau, {\em On the Existence of a Complete K\"ahler Metric on Non-Compact Complex Manifold and the Regularity of Fefferman's Equation,} CPAM, Vol. XXXIII, 507-544(1980)

\bibitem[CHS26]{ChenHongyu}H.Chen, J.Hu, L.Sheng, {\em Power Convexity of Solutions to Complex Monge-Amp\`ere Equation} (to appear)


\bibitem[F76]{Fefferman1976Ann}C. L. Fefferman, {\em Monge-Ampère equations, the Bergman kernel, and geometry of pseudoconvex domains}. 
Ann. of Math. (2), 103(2):395–416, 1976.



\bibitem[GT01]{GT} D. Gilbarg, N. Trudinger, {\em Elliptic Partial Differential Equations of Second Order,} Classics in Mathematics, Springer-Verlag, 2001. 
	\bibitem[GP12]{GuanPhong12MaxRank} P. Guan, D. H. Phong, {\em A Maximum Rank Problem for Degenerate Elliptic Fully Nonlinear Equations}. Math. Ann. {354} (2012), no. 1, 147--169.
	\bibitem[GX13]{GuanXu_convexity} 
	P. Guan, L. Xu, {\em Convexity estimates for level sets of quasiconcave solutions to fully nonlinear elliptic equations}. Journal für die reine und angewandte Mathematik (Crelles Journal), vol. 2013, no. 680, 2013, pp. 41-67.
	
	\bibitem[HJ23]{HanJiang_MinimalSurface}  Q. Han, X. Jiang, {\em Boundary Regularity for Minimal Graphs in the Hyperbolic Space}. Journal für die reine und angewandte Mathematik (Crelles Journal) 2023 (2014): 239 - 272.

\bibitem[HJ24]{Han_Jiang_Expansion} Q. Han, X. Jiang, {\em Asymptotics and Convergence for the Complex Monge-Amp\`ere Equation,} Annals of PDE(2024) 10:8.
Nov. 2011, Adv. Math. Vol. 230, Iss. 4–6, Jul.–Aug. 2012, Pages 2428-2456.

\bibitem[HJ13]{HornMatrix} R. Horn, C. Johnson, Matrix Analysis, Cambridge Univ. Press, Cambridge, Second Edition 2013.

	\bibitem[H21]{HuIMRN} J. Hu, {\em An Obstacle for Higher Regularity of Geodesics in the Space of K\"ahler Potentials.} International Mathematics Research Notices, Volume 2021, Issue 15, August 2021, Pages 11493–11513.
	\bibitem[H24A]{HuC2Perturb} J. Hu, {\em A Metric Lower Bound Estimate for  Geodesics in the Space of K\"ahler Potentials.} The Journal of Geometric Analysis (2024) 34:225.
	\bibitem[H24B]{HuCConvexity}J. Hu, {\em A Maximum Rank Theorem for Solutions to the Homogenous Complex Monge–Amp\`ere Equation in a $\EC$-Convex Ring.} Calc.Var.(2024) 63:166.
	\bibitem[H25]{Hu22Nov}J. Hu, {\em The Preservation of Convexity by Geodesics in the Space of K\"ahler Potentials on Complex Affine Manifolds}. 	Math. Annalen (2025) 393:1635–1681.


\bibitem[K87]{Korevaar_Lewis} N.J.Korevaar, J.Lewis, {\em Convex 
	Solutions of Certain Elliptic Equations Have Constant Rank Hessians}. Arch. Rational Mech. Anal. 97, 19–32 (1987)



 
\bibitem[LM82]{Lee_Melrose_Acta} J. M. Lee and R. Melrose, Boundary Behaviour of the Complex Monge-Amp\`ere Equation, Acta Math., 148 (1982), 159–192.
 \bibitem[L09]{LQun_Complex} Li, Q,{\em Constant rank theorem in complex variables}, Indiana Univ. Math. J. 58 (2009),
 no. 3, 1235–1256.
 
 

   \bibitem[M16]{XinanMa_Chinese} X. Ma, {\em The Convexity of the Solution of Elliptic and Parabolic Partial Differential Equations}, College Mathematics. Vol. 32, No. 5, Oct. 2016, 1-17. (in Chinese)
   \bibitem[MO10]{Xinan_Ou_Summary_English}X.Ma, Q. Ou, {\em The Convexity of Level Sets for Solutions
   	to Partial Differential Equations.} Trends in
   Partial Differential Equations,
   Advanced Lectures in Mathematics
   Volume X, pp. 295–322. International Press of Boston, Inc. 2010.

\bibitem[MX08]{MXu_PowerConvexity} X. Ma, L. Xu, {\em The Convexity of Solution of a class Hessian Equation	in Bounded Convex Domain in $\ER^3$.} Journal of Functional Analysis 255 (2008) 1713–1723.
   
   \bibitem[S15]{Shishujun_Green_LevelSets}S. Shi, {\em Convexity Estimates for the Green’s Function}. Calc. Var., 53, 675–688 (2015).


   \bibitem[SW16]{Garbo_Weinkove_ConstantRank} G. Sz\'ekelyhidi, B. Weinkove,  {\em On a Constant Rank Theorem for Nonlinear Elliptic PDEs}. Discrete Contin. Dyn. Syst. { 36} (2016), no. 11, 6523--6532.
   

\end{thebibliography}
\end{document}